\documentclass[reqno]{amsart}
\usepackage{amssymb}
\usepackage{ifpdf}
\ifpdf
  \usepackage[hyperindex,pagebackref]{hyperref}
\else
  \expandafter\ifx\csname dvipdfm\endcsname\relax
    \usepackage[hypertex,hyperindex,pagebackref]{hyperref}
  \else
    \usepackage[dvipdfm,hyperindex,pagebackref]{hyperref}
  \fi
\fi
\allowdisplaybreaks[4]
\theoremstyle{plain}
\newtheorem{thm}{Theorem}

\theoremstyle{remark}
\newtheorem{rem}{Remark}
\DeclareMathOperator{\re}{Re}
\DeclareMathOperator{\td}{d\mspace{-2mu}}
\numberwithin{equation}{section}
\date{Commenced on 13 April 2011 in Tianjin Polytechnic University}
\date{}

\begin{document}

\title[Complete monotonicity of function involving gamma function]
{Complete monotonicity of a function involving the gamma function and applications}

\author[F. Qi]{Feng Qi}
\address{School of Mathematics and Informatics\\ Henan Polytechnic University\\ Jiaozuo City, Henan Province, 454010\\ China; Department of Mathematics\\ College of Science\\ Tianjin Polytechnic University\\ Tianjin City, 300160\\ China}
\email{\href{mailto: F. Qi <qifeng618@gmail.com>}{qifeng618@gmail.com}, \href{mailto: F. Qi <qifeng618@hotmail.com>}{qifeng618@hotmail.com}, \href{mailto: F. Qi <qifeng618@qq.com>}{qifeng618@qq.com}}
\urladdr{\url{http://qifeng618.wordpress.com}}

\begin{abstract}
In the article we present necessary and sufficient conditions for a function involving the logarithm of the gamma function to be completely monotonic and apply these results to bound the gamma function $\Gamma(x)$, the $n$-th harmonic number $\sum_{k=1}^n\frac1k$, and the factorial $n!$.
\end{abstract}

\keywords{Complete monotonicity, logarithmically completely monotonic function, gamma function, inequality, necessary and sufficient condition, harmonic number, factorial, application}

\subjclass[2010]{Primary 26A48, 33B15; Secondary 26D15}

\thanks{The author was supported in part by the Science Foundation of Tianjin Polytechnic University}

\maketitle

\section{Introduction}

We recall from \cite[Chapter~XIII]{mpf-1993}, \cite[Chapter~1]{Schilling-Song-Vondracek-2010} and \cite[Chapter~IV]{widder} that a function
$f$ is said to be completely monotonic on an interval $I$ if $f$ has derivatives of all
orders on $I$ and
\begin{equation}\label{CM-dfn}
0\le(-1)^{n}f^{(n)}(x)<\infty
\end{equation}
for $x\in I$ and $n\ge0$.
\par
We also recall from~\cite[p.~254, 6.1.1]{abram} that the classical Euler gamma function $\Gamma(x)$ may be defined by
\begin{equation}\label{gamma-dfn}
\Gamma(x)=\int^\infty_0t^{x-1} e^{-t}\td t,\quad x>0.
\end{equation}
The logarithmic derivative of $\Gamma(x)$, denoted by $\psi(x)=\frac{\Gamma'(x)}{\Gamma(x)}$, is called the psi or di-gamma function, and the derivatives $\psi^{(i)}(x)$ for $i\in\mathbb{N}$ are respectively called the polygamma functions. They are a series of important special functions and have much extensive applications in many branches such as statistics, probability, number theory, theory of $0$-$1$ matrices, graph theory, combinatorics, physics, engineering, and other mathematical sciences.
\par
Using Hermite-Hadamard's inequality (see~\cite{difference-hermite-hadamard.tex, correction-to-sandor.tex-ijams, correction-to-sandor.tex-octogon}), the double inequality
\begin{multline}\label{p.236-Theorem-1}
  \biggl(x-\frac12\biggr)\biggl[\ln\biggl(x-\frac12\biggr)-1\biggr]+\ln\sqrt{2\pi}\, -\frac1{24(x-1)}\le\ln\Gamma(x)\\*
  \le\biggl(x-\frac12\biggr)\biggl[\ln\biggl(x-\frac12\biggr)-1\biggr] +\ln\sqrt{2\pi}\,-\frac1{24\bigl(\sqrt{x^2+x+1/2}\,-1/2\bigr)}
\end{multline}
for $x>1$ was obtained in~\cite[p.~236, Theorem~1]{Bukac-Buric-Elezovic-MIA-2011}.
\par
In~\cite[p.~1774, Theorem~2.3]{Sevli-Batir-MCM-2011}, the function
\begin{equation}
  H(x)=\ln\Gamma(x+1)-\biggl(x+\frac12\biggr)\ln\biggl(x+\frac12\biggr)+x+\frac12-\frac12\ln(2\pi) +\frac1{24(x+1/2)}
\end{equation}
was proved to be completely monotonic on $(0,\infty)$. From this it was deduced in~\cite[p.~1775, Corollary~2.4]{Sevli-Batir-MCM-2011} that the double inequality
\begin{multline}\label{p-1775-Corollary-2.4}
  \alpha\biggl(\frac{x+1/2}{e}\biggr)^{x+1/2}e^{-1/24(x+1/2)}<\Gamma(x+1) \\*
  \le\beta\biggl(\frac{x+1/2}{e}\biggr)^{x+1/2}e^{-1/24(x+1/2)}
\end{multline}
holds for $x>0$, where $\alpha=\sqrt{2\pi}\,=2.50\dotsm$ and $\beta=\sqrt2\,e^{7/12}=2.53\dotsm$ are the best possible constants.
\par
We observe that, by taking the natural exponentials on all sides of \eqref{p.236-Theorem-1} and replacing $x$ by $x+1$, the inequality~\eqref{p.236-Theorem-1} may be rewritten as
\begin{multline}\label{p.236-Theorem-1-rew}
  \sqrt{2\pi}\,\biggl(\frac{x+1/2}{e}\biggr)^{x+1/2}e^{-1/24x}<\Gamma(x+1) \\ \le\sqrt{2\pi}\,\biggl(\frac{x+1/2}{e}\biggr)^{x+1/2}e^{-1/24\bigl(\sqrt{x^2+3x+5/2}\,-1/2\bigr)},\quad x>0.
\end{multline}
Hence, it is clear that the left hand side inequality in~\eqref{p-1775-Corollary-2.4} is stronger than the corresponding one in~\eqref{p.236-Theorem-1} or, equivalently, \eqref{p.236-Theorem-1-rew}. But, when $x\ge1$, the right hand side inequality in~\eqref{p-1775-Corollary-2.4} is weaker than the corresponding one in~\eqref{p.236-Theorem-1} or, equivalently, \eqref{p.236-Theorem-1-rew}.
\par
For $\lambda\ge0$ and $x\in(0,\infty)$, let
\begin{equation}\label{H-lambda(x)-dfn}
H_\lambda(x)=\ln\Gamma(x+1)-\biggl(x+\frac12\biggr)\ln\biggl(x+\frac12\biggr)+x+\frac12-\ln\sqrt{2\pi}\, +\frac1{24(x+\lambda)}.
\end{equation}
The first aim of this paper is to find necessary and sufficient conditions for the functions $\pm H_\lambda(x)$ to be completely monotonic on $(0,\infty)$.
The second aim is to apply the complete monotonicity of $\pm H_\lambda(x)$ to establish inequalities for bounding the gamma function, the $n$-th harmonic number $\sum_{k=1}^n\frac1k$, and the factorial $n!$.

\section{Necessary and sufficient conditions}

Our main results are necessary and sufficient conditions for the functions $\pm H_\lambda(x)$ to be completely monotonic on $(0,\infty)$, which can be stated in the following theorem.

\begin{thm}\label{Bukac-Sevli-Gamma-thm-1}
For $x\in(0,\infty)$ and $\lambda\ge0$,
\begin{enumerate}
\item\label{Bukac-Sevli-Gamma-thm-1-item-1}
if and only if $0\le\lambda\le\frac12$, the function $H_\lambda(x)$ is completely monotonic;
\item\label{Bukac-Sevli-Gamma-thm-1-item-2}
if
\begin{equation}\label{lamda-lower-abs}
\lambda\ge-\inf_{t\in(0,\infty)}\biggl\{\frac1t\ln\biggl[\frac{24}{t^2}\biggl(\frac1{e^{t/2}} -\frac{t}{e^t-1}\biggr)\biggr]\biggr\},
\end{equation}
the function $-H_\lambda(x)$ is completely monotonic;
\item\label{Bukac-Sevli-Gamma-thm-1-item-3}
if $\lambda\ge\frac32$, a special case of the inequality~\eqref{lamda-lower-abs}, the function $-H_\lambda(x)$ is completely monotonic on $(0,\infty)$.
\end{enumerate}
\end{thm}

\subsubsection*{Proof of the necessary condition in (\ref{Bukac-Sevli-Gamma-thm-1-item-1}) of Theorem~\ref{Bukac-Sevli-Gamma-thm-1}}
If the function $H_\lambda(x)$ is completely monotonic on $(0,\infty)$ for $\lambda\ge0$, then $H_\lambda(x)\ge0$, which may be rearranged as
\begin{equation*}
  \lambda\le-x-\frac1{24f(x)},
\end{equation*}
where
\begin{equation}\label{f(x)-dfn}
f(x)=\ln\Gamma(x+1)-\biggl(x+\frac12\biggr)\ln\biggl(x+\frac12\biggr) +x+\frac12-\frac12\ln(2\pi),\quad x>0.
\end{equation}
Since
\begin{gather}
\ln\Gamma(z)\sim\biggl(z-\frac12\biggr)\ln z-z+\frac12\ln(2\pi)+\frac1{12z}-\frac1{360z^3} +\frac{1}{1260z^5}-\dotsm
\end{gather}
as $z\to\infty$ in $\lvert \arg z\rvert<\pi$, see~\cite[p.~257, 6.1.41]{abram}, we have
\begin{equation}\label{f(x)-sim}
  f(x)\sim\biggl(x+\frac12\biggr)\ln\frac{x+1}{x+1/2}-\frac12+\frac1{12(x+1)}+O\biggl(\frac1{x^2}\biggr)
\end{equation}
as $x\to\infty$. As a result, we have
$$
-x-\frac1{24f(x)}\sim-\frac12\cdot\frac{h_1(x)}{h_2(x)}
$$
as $x\to\infty$, where
\begin{equation*}
  h_1(x)=12x\bigl(2x^2+3x+1\bigr) \ln\frac{x+1}{x+1/2}-12 x^2-9 x+1+O(1)\to\frac12
\end{equation*}
and
$$
h_2(x)=6\bigl(2x^2+3x +1\bigr)\ln\frac{x+1}{x+1/2}-6x-5+O\biggl(\frac1{x}\biggr) \to-\frac12
$$
as $x\to\infty$. So
$$
-x-\frac1{24f(x)}\to\frac12,\quad x\to\infty.
$$
This means that the necessary condition for the function $H_\lambda(x)$ to be completely monotonic on $(0,\infty)$ is $\lambda\le\frac12$.

\subsubsection*{First proof of the sufficient condition in (\ref{Bukac-Sevli-Gamma-thm-1-item-1}) of Theorem~\ref{Bukac-Sevli-Gamma-thm-1}}
When $\lambda<\frac12$, the function $H_\lambda(x)$ is
\begin{equation*}
\begin{split}
H_\lambda(x)&=H_{1/2}(x)+\frac1{24(x+\lambda)}-\frac1{24(x+1/2)} \\ &=H_{1/2}(x)+\frac1{24}\int_\lambda^{1/2}\frac1{(x+t)^2}\td t.
\end{split}
\end{equation*}
Since the function $H_{1/2}(x)=H(x)$ is completely monotonic on $(0,\infty)$, see~\cite[p.~1774, Theorem~2.3]{Sevli-Batir-MCM-2011}, and the integral $\int_\lambda^{1/2}\frac1{(x+t)^2}\td t$ is clearly completely monotonic on $(-\lambda,\infty)\supset(0,\infty)$, also since the product and sum of finite completely monotonic functions are completely monotonic on their common domain, it follows immediately that the function $H_\lambda(x)$ is completely monotonic on $(0,\infty)$ if $\lambda\le\frac12$.

\subsubsection*{Second proof of the sufficient condition in (\ref{Bukac-Sevli-Gamma-thm-1-item-1}) of Theorem~\ref{Bukac-Sevli-Gamma-thm-1}}
The famous Binet's first formula of $\ln\Gamma(x)$ for $x>0$ is given by
\begin{equation}\label{ebinet}
\ln \Gamma(x)= \left(x-\frac{1}{2}\right)\ln x-x+\ln \sqrt{2\pi}\,+\theta(x),
\end{equation}
where
\begin{equation}\label{ebinet1}
\theta(x)=\int_{0}^{\infty}\biggl(\frac{1}{e^{t}-1}-\frac{1}{t}
+\frac{1}{2}\biggr)\frac{e^{-xt}}{t}\td t
\end{equation}
for $x>0$ is called the remainder of Binet's first formula for the logarithm of the gamma function, see \cite[p.~11]{magnus} or \cite[p.~462]{Extended-Binet-remiander-comp.tex}. The formulas
\begin{equation}\label{Gamma(z)=k-z-int}
\Gamma(z)=k^z\int_0^\infty t^{z-1}e^{-kt}\td t
\end{equation}
and
\begin{equation}\label{ln-frac}
\ln\frac{b}a=\int_0^\infty\frac{e^{-au}-e^{-bu}}u\td u
\end{equation}
for $\re z>0$, $\re k>0$, $a>0$ and $b>0$ can be found in~\cite[p.~255, 6.1.1 and p.~230, 5.1.32]{abram}.
Utilizing these formulas yields
\begin{align}\label{f(x)-integral}
f(x)&=\biggl(x+\frac12\biggr)\ln\frac{x}{x+1/2}+\frac12+\theta(x),\\
\begin{split}\label{f(x)-integral-der}
f'(x)&=\frac1{2x}+\ln\frac{x}{x+1/2}+\theta'(x) \\ &=\int_0^\infty\biggl(\frac{e^{-t/2}}t-\frac{1}{e^{t}-1}\biggr)e^{-xt}\td t,
\end{split}\\
\begin{split}\label{H-der-lanbda}
H_\lambda'(x)&=f'(x)-\frac1{24(x+\lambda)^2}\\
&=\int_0^\infty\biggl(\frac{e^{-t/2}}t-\frac{1}{e^{t}-1} -\frac{te^{-\lambda t}}{24}\biggr)e^{-xt}\td t.
\end{split}
\end{align}
From \eqref{H-lambda(x)-dfn}, \eqref{f(x)-dfn} and \eqref{f(x)-integral}, it is easy to see that
\begin{equation}\label{lim-infty-H-lambda}
\lim_{x\to\infty}H_\lambda(x)=\lim_{x\to\infty}f(x)+\lim_{x\to\infty}\frac1{24(x+\lambda)}=0.
\end{equation}
Therefore, in order to prove the complete monotonicity of $\mp H_\lambda(x)$, it suffices to show $\pm H_\lambda'(x)$ is completely monotonic on $(0,\infty)$. For this, it is sufficient to have
$$
\frac{e^{-t/2}}t-\frac{1}{e^{t}-1} -\frac{te^{-\lambda t}}{24}\gtreqless0
$$
for all $t\in(0,\infty)$, which is equivalent to
\begin{equation}\label{lambda-gtreqless}
\lambda\gtreqless-\frac1t\ln\biggl[\frac{24}t\biggl(\frac{e^{-t/2}}t-\frac1{e^t-1}\biggr)\biggr]
\end{equation}
for all $t\in(0,\infty)$.
\par
We claim that
$$
-\frac1t\ln\biggl[\frac{24}t\biggl(\frac{e^{-t/2}}t-\frac1{e^t-1}\biggr)\biggr]\ge\frac12
$$
for $t\in(0,\infty)$. In fact, this inequality can be reduced to
\begin{equation}\label{upper-exp-24}
\frac{24}t\biggl(\frac{e^{-t/2}}t-\frac1{e^t-1}\biggr)\le e^{-t/2},
\end{equation}
equivalently,
\begin{gather*}
\bigl(t^2-24\bigr)e^t+24t e^{t/2}-t^2+24 =\sum_{k=5}^\infty\bigl\{[k(k-1)-24]2^k+{48k}\bigr\}\frac{t^k}{k!2^k}\\
\begin{split}
&=\frac7{24}t^5+\sum_{k=6}^\infty\bigl\{[k(k-1)-24]2^k+{48k}\bigr\}\frac{t^k}{k!2^k} \\ &\ge\frac7{24}t^5+\sum_{k=6}^\infty\bigl\{[k(k-1)-24](1+k)+{48k}\bigr\}\frac{t^k}{k!2^k}\\
&\ge\frac7{24}t^5+\sum_{k=6}^\infty\bigl(k^3+23k-24\bigr)\frac{t^k}{k!2^k}\\
&\ge0.
\end{split}
\end{gather*}
Thus, when $0\le\lambda\le\frac12$, the function $H_\lambda(x)$ is completely monotonic on $(0,\infty)$.

\subsubsection*{Proof of (\ref{Bukac-Sevli-Gamma-thm-1-item-2}) in Theorem~\ref{Bukac-Sevli-Gamma-thm-1}}
This follows from the inequality with the sign $\ge$ in~\eqref{lambda-gtreqless}.

\subsubsection*{Proof of (\ref{Bukac-Sevli-Gamma-thm-1-item-3}) in Theorem~\ref{Bukac-Sevli-Gamma-thm-1}}
Suppose that
\begin{equation}\label{tau-upper-ask}
-\frac1t\ln\biggl[\frac{24}t\biggl(\frac{e^{-t/2}}t-\frac1{e^t-1}\biggr)\biggr]\le\lambda
\end{equation}
for $t\in(0,\infty)$. Then
$$
\frac{24}t\biggl(\frac{e^{-t/2}}t-\frac1{e^t-1}\biggr)\ge e^{-\lambda t},
$$
which can be rewritten as
$$
24e^{\lambda t}\bigl(e^t-te^{t/2}-1\bigr)\ge t^2e^{t/2}(e^t-1).
$$
Expanding at $t=0$ the functions on both sides of the above inequality into power series yields
\begin{equation}\label{ineq-series-expansion}
24\sum_{k=3}^\infty\Biggl[(\lambda+1)^k-\lambda^k-k\biggl(\lambda+\frac12\biggr)^{k-1}\Biggr]\frac{t^k}{k!} \ge\sum_{k=3}^\infty\Biggl[\biggl(\frac32\biggr)^{k-2}-\biggl(\frac12\biggr)^{k-2}\Biggr]\frac{t^k}{(k-2)!}.
\end{equation}
Let $h_{k;\lambda}(u)=(\lambda+u)^{k-1}$. Then, by the left hand side inequality in the double integral inequality
\begin{equation}\label{ps1}
\frac{(b-a)^2}{24}m\le\frac1{b-a}\int_a^bf(t)\td t -f\left(\frac{a+b}2\right)\le\frac{(b-a)^2}{24}M,
\end{equation}
where $f:[a,b]\to\mathbb{R}$ be a twice differentiable mapping and $m\le f''(t)\le M$ for all $t\in(a,b)$, see \cite{psever1,psever2} and~\cite[p.~236, Theorem~A]{difference-hermite-hadamard.tex}, we obtain
\begin{align*}
(\lambda+1)^k-\lambda^k-k\biggl(\lambda+\frac12\biggr)^{k-1}& =k\Biggl[\frac1{1-0}\int_0^1h_{k;\lambda}(u)\td u-h_{k;\lambda}\biggl(\frac{0+1}2\biggr)\Biggr]\\
&\ge k\cdot\frac{(1-0)^2}{24}\inf_{u\in(0,1)}h_{k;\lambda}''(u)\\
&=\frac1{24}k(k-1)(k-2)\lambda^{k-3}.
\end{align*}
Therefore, in order to have the inequality~\eqref{tau-upper-ask} hold, it is sufficient to make
$$
(k-2)\lambda^{k-3}\ge\biggl(\frac32\biggr)^{k-2}-\biggl(\frac12\biggr)^{k-2},
$$
which is equivalent to
$$
\lambda\ge\sqrt[k-3]{\frac1{k-2}\Biggl[\biggl(\frac32\biggr)^{k-2}-\biggl(\frac12\biggr)^{k-2}\Biggr]}
$$
for all $k\ge3$. Since
$$
\sqrt[k-3]{\frac1{k-2}\Biggl[\biggl(\frac32\biggr)^{k-2}-\biggl(\frac12\biggr)^{k-2}\Biggr]}\, \le\sqrt[k-3]{\frac1{k-2}\biggl(\frac32\biggr)^{k-2}}\,\to\frac32
$$
as $k\to\infty$, it follows that when $\lambda\ge\frac32$ the inequality~\eqref{tau-upper-ask} holds. This means that when $\lambda\ge\frac32$ the negative of the function~\eqref{H-der-lanbda} is completely monotonic on $(0,\infty)$. The proof of (\ref{Bukac-Sevli-Gamma-thm-1-item-2}) in Theorem~\ref{Bukac-Sevli-Gamma-thm-1} is complete.

\section{Applications}
In this section, we apply the complete monotonicity of $\pm H_\lambda(x)$ to establish inequalities for bounding the gamma function, the $n$-th harmonic number $\sum_{k=1}^n\frac1k$, and the factorial $n!$.

\begin{thm}\label{bukac-app-thm1}
For $x\in(0,\infty)$, the gamma function $\Gamma(x+1)$ can be bounded by
\begin{multline}\label{bukac-app-thm1-ineq-1}
\sqrt{2\pi}\,\biggl(\frac{x+1/2}{e}\biggr)^{x+1/2}\exp\biggl(-\frac1{24(x+1/2)}\biggr)<\Gamma(x+1)\\*
<\sqrt{2\pi}\,\biggl(\frac{x+1/2}{e}\biggr)^{x+1/2} \exp\biggl(\frac1{24}\biggl(\frac{2x}{x+1/2}-12(\ln\pi-1)\biggr)\biggr)
\end{multline}
and
\begin{multline}\label{bukac-app-thm1-ineq-2}
\sqrt{2\pi}\,\biggl(\frac{x+1/2}{e}\biggr)^{x+1/2} \exp\biggl(\frac1{24}\biggl[\frac{2x}{3(x+3/2)}-12(\ln\pi-1)\biggr]\biggr)<\Gamma(x+1)\\*
<\sqrt{2\pi}\,\biggl(\frac{x+1/2}{e}\biggr)^{x+1/2}\exp\biggl(-\frac1{24(x+3/2)}\biggr).
\end{multline}
\end{thm}

\begin{proof}
By~\eqref{f(x)-sim}, it is easy to see that
\begin{equation}\label{lim-H-lambda-infty}
\lim_{x\to\infty}H_\lambda(x)=0.
\end{equation}
Moreover, it is immediate that
$$
\lim_{x\to0^+}H_\lambda(x)=\frac{1}{24} \biggl(\frac{1}{\lambda}+12-12 \ln \pi\biggr).
$$
By Theorem~\ref{Bukac-Sevli-Gamma-thm-1}, it readily follows that when and only when $0\le\lambda\le\frac12$ the function $H_\lambda(x)$ is decreasing on $(0,\infty)$. So, we have
$$
0<H_\lambda(x)<\frac{1}{24} \biggl(\frac{1}{\lambda}+12-12 \ln \pi\biggr),
$$
that is,
\begin{multline}\label{bukac-app-thm1-ineq}
\sqrt{2\pi}\,\biggl(\frac{x+1/2}{e}\biggr)^{x+1/2}\exp\biggl(-\frac1{24(x+\lambda)}\biggr)<\Gamma(x+1)\\*
<\sqrt{2\pi}\,\biggl(\frac{x+1/2}{e}\biggr)^{x+1/2} \exp\biggl(\frac1{24}\biggl(\frac1\lambda+12-12\ln\pi-\frac1{x+\lambda}\biggr)\biggr)
\end{multline}
for $0\le\lambda\le\frac12$ and $x\in(0,\infty)$. The inequality~\eqref{bukac-app-thm1-ineq-1} is proved.
\par
Similarly, when $\lambda\ge\frac32$, the function $H_\lambda(x)$ is increasing on $(0,\infty)$, and the inequality~\eqref{bukac-app-thm1-ineq} is reversed on $(0,\infty)$ for $\lambda\ge\frac32$. The inequality~\eqref{bukac-app-thm1-ineq-2} follows.
\end{proof}

\begin{thm}\label{bukac-app-thm2}
For $n\in\mathbb{N}$, the $n$-th harmonic number $\sum_{k=1}^n\frac1k$ can be bounded by
\begin{multline}\label{bukac-app-thm2-ineq1}
\ln\biggl(n+\frac{1}{2}\biggr)+\frac1{24(n+1/2)^2}+1-\ln\frac32-\frac1{54} \le\sum_{k=1}^n\frac1k\\*
<\ln\biggl(n+\frac{1}{2}\biggr)+\frac1{24(n+1/2)^2}+\gamma
\end{multline}
and
\begin{multline}\label{bukac-app-thm2-ineq2}
\ln\biggl(n+\frac{1}{2}\biggr)+\frac1{24(n+3/2)^2}+\gamma<\sum_{k=1}^n\frac1k\\*
\le\ln\biggl(n+\frac{1}{2}\biggr)+\frac1{24(n+3/2)^2}+1-\ln\frac32-\frac1{90},
\end{multline}
where $\gamma=0.577\dotsm$ stands for Euler-Mascheroni's constant.
\end{thm}

\begin{proof}
By Theorem~\ref{Bukac-Sevli-Gamma-thm-1} and the definition of completely monotonic functions, it follows that
\begin{enumerate}
\item
when $0\le\lambda\le\frac12$, the function $H_\lambda'(x)$ is increasing on $(0,\infty)$,
\item
when $\lambda\ge\frac32$, the function $H_\lambda'(x)$ is decreasing on $(0,\infty)$.
\end{enumerate}
Since
$$
H_\lambda'(x)=\psi(x+1)-\ln \biggl(x+\frac{1}{2}\biggr)-\frac{1}{24 (x+\lambda)^2}
$$
and $\lim_{x\to\infty}H_\lambda'(x)=0$, it follows readily that
\begin{multline}\label{bukac-app-thm2-ineq1-pre}
\ln\biggl(x+\frac{1}{2}\biggr)+\frac1{24(x+1/2)^2}+1-\gamma-\ln\frac32-\frac1{54}\le\psi(x+1)\\*
<\ln\biggl(x+\frac{1}{2}\biggr)+\frac1{24(x+1/2)^2}
\end{multline}
and
\begin{multline}\label{bukac-app-thm2-ineq2-pre}
\ln\biggl(x+\frac{1}{2}\biggr)+\frac1{24(x+3/2)^2}<\psi(x+1)\\*
\le\ln\biggl(x+\frac{1}{2}\biggr)+\frac1{24(x+3/2)^2}+1-\gamma-\ln\frac32-\frac1{90}.
\end{multline}
for $x\in[1,\infty)$. Taking $x=n$ and using
\begin{equation}
\psi(n+1)=\sum_{k=1}^n\frac1k-\gamma,
\end{equation}
see~\cite[p.~258, 6.3.2]{abram}, in~\eqref{bukac-app-thm2-ineq1-pre} and~\eqref{bukac-app-thm2-ineq2-pre} give inequalities~\eqref{bukac-app-thm2-ineq1} and~\eqref{bukac-app-thm2-ineq2}.
\end{proof}

We recall from~\cite{Atanassov, compmon2} that a function $f$ is said to be logarithmically completely monotonic on an interval $I\subseteq\mathbb{R}$ if it has derivatives of all orders on $I$ and its logarithm $\ln f$ satisfies
\begin{equation}\label{lcm-dfn}
(-1)^k[\ln f(x)]^{(k)}\ge0
\end{equation}
for $k\in\mathbb{N}$ on $I$.

\begin{thm}\label{bukac-app-LCM-thm}
For $x\in(0,\infty)$ and $\lambda\ge0$, let
\begin{equation}
  G_\lambda(x)=\frac{e^{x}\Gamma(x+1)}{(x+1/2)^{x+1/2}}\exp\frac1{24(x+\lambda)}.
\end{equation}
Then the function $G_\lambda(x)$ has the following properties:
\begin{enumerate}
\item
if and only if $0\le\lambda\le\frac12$, the function $G_\lambda(x)$ is logarithmically completely monotonic on $(0,\infty)$;
\item
if the inequality~\eqref{lamda-lower-abs} is valid, the reciprocal of the function $G_\lambda(x)$ is logarithmically completely monotonic on $(0,\infty)$;
\item
if $\lambda\ge\frac32$, a special case of the inequality~\eqref{lamda-lower-abs}, the reciprocal of the function $G_\lambda(x)$ is logarithmically completely monotonic on $(0,\infty)$.
\end{enumerate}
\end{thm}

\begin{proof}
This follows from the obvious fact that
$$
\ln G_\lambda(x)=H_\lambda(x)-\frac{1-\ln(2\pi)}2
$$
and the definition of logarithmically completely monotonic functions.
\end{proof}

\begin{thm}\label{bukac-factirial-ineq-thm}
For $n\in\mathbb{N}$, the factorial $n!$ can be bounded by
\begin{multline}\label{bukac-thm1-ineq-1n!}
\sqrt{2\pi}\,\biggl(\frac{n+1/2}{e}\biggr)^{n+1/2}\exp\biggl(-\frac1{24(n+1/2)}\biggr)<n!\\*
\le\sqrt{2\pi}\,\biggl(\frac{n+1/2}{e}\biggr)^{n+1/2} \exp\biggl(\frac1{24}\biggl[12\biggl(3-\ln\pi+\ln\frac{4}{27}\biggr) -\frac1{3(n+1/2)}\biggr]\biggr)
\end{multline}
and
\begin{multline}\label{bukac-thm1-ineq-2n!}
\sqrt{2\pi}\,\biggl(\frac{n+1/2}{e}\biggr)^{n+1/2} \exp\biggl(\frac1{24}\biggl[12\biggl(3-\ln\pi+\ln\frac{4}{27}\biggr) +\frac1{5(n+3/2)}\biggr]\biggr)\le n!\\* <\sqrt{2\pi}\,\biggl(\frac{n+1/2}{e}\biggr)^{n+1/2}\exp\biggl(-\frac1{24(n+3/2)}\biggr).
\end{multline}
\end{thm}

\begin{proof}
Combining~\eqref{lim-H-lambda-infty} and
$$
H_\lambda(1)=\frac{1}{24} \biggl[\frac{1}{\lambda+1}+36-12 \ln(2\pi)-36\ln\frac{3}{2}\biggr],
$$
with Theorem~\ref{Bukac-Sevli-Gamma-thm-1} and the proof of Theorem~\ref{bukac-app-thm1} reveals
\begin{multline}\label{bukac-app-thm1-ineq-1n!}
\sqrt{2\pi}\,\biggl(\frac{x+1/2}{e}\biggr)^{x+1/2}\exp\biggl(-\frac1{24(x+1/2)}\biggr)<\Gamma(x+1)\\*
\le\sqrt{2\pi}\,\biggl(\frac{x+1/2}{e}\biggr)^{x+1/2} \exp\biggl(\frac1{24}\biggl[12\biggl(3-\ln\pi+\ln\frac{4}{27}\biggr) -\frac1{3(x+1/2)}\biggr]\biggr)
\end{multline}
and
\begin{multline}\label{bukac-app-thm1-ineq-2n!}
\sqrt{2\pi}\,\biggl(\frac{x+1/2}{e}\biggr)^{x+1/2} \exp\biggl(\frac1{24}\biggl[12\biggl(3-\ln\pi+\ln\frac{4}{27}\biggr) +\frac1{5(x+3/2)}\biggr]\biggr)\le\\* \Gamma(x+1)
<\sqrt{2\pi}\,\biggl(\frac{x+1/2}{e}\biggr)^{x+1/2}\exp\biggl(-\frac1{24(x+3/2)}\biggr).
\end{multline}
Letting $x=n$ and using $\Gamma(n+1)=n!$ in~\eqref{bukac-app-thm1-ineq-1n!} and~\eqref{bukac-app-thm1-ineq-2n!} leads to inequalities~\eqref{bukac-thm1-ineq-1n!} and~\eqref{bukac-thm1-ineq-2n!}. The proof is complete.
\end{proof}

\section{Remarks}

In this section, we would like to comment some results  above-presented.

\begin{rem}
The inequality~\eqref{upper-exp-24} and the inequality~\eqref{tau-upper-ask} for $\lambda\ge\frac32$ may be rearranged as the double inequality
\begin{equation}\label{exp-bernoulli-ineq}
\frac1{e^{x/2}}-\frac{x^2}{24e^{\beta x}}\le\frac{x}{e^x-1}\le\frac1{e^{x/2}}-\frac{x^2}{24e^{\alpha x}}
\end{equation}
on $(0,\infty)$, where $\alpha=\frac32$ and $\beta=\frac12$. This improves the right hand side and partially improves the left hand side of the double inequality
\begin{equation}\label{expin}
\frac1{e^x}<\frac{x}{e^x-1}<\frac1{e^{x/2}},\quad x>0
\end{equation}
in~\cite[p.~2550, Proposition~4.1]{mathieuijms.tex}.
\par
We guess that the scalar $\alpha=\frac32$ in~\eqref{exp-bernoulli-ineq} can be replaced by a smaller number, for example, $1$, but the constant $\beta=\frac12$ in~\eqref{exp-bernoulli-ineq} is the best possible.
\par
In~\cite{e^x-beograd}, some related inequalities for the exponential function $e^x$ were constructed.
\par
In a subsequent paper, we will refine the right hand side of the double inequality~\eqref{exp-bernoulli-ineq} and employ it to strengthen double inequalities for bounding Mathieu's series
\begin{equation}\label{mathieu-series}
S(r)=\sum_{n=1}^\infty\frac{2n}{(n^2+r^2)^2},\quad r>0
\end{equation}
and the like. For more information on bounding Mathieu type series, please refer to~\cite{Hoorfar-Qi.tex, mathieu-rostock, mathieuijms.tex} and closely related references therein.
\end{rem}

\begin{rem}
There have been plenty of references devoted to bounding the $n$-th harmonic number $\sum_{k=1}^n\frac1k$ for $n\in\mathbb{N}$, for example, \cite{Global-JAMMS-2008, harmonic-number-refine-chen.tex, property-psi.tex, qi-cui-jmaa} and closely related references therein.
\end{rem}

\begin{rem}
Several inequalities for bounding the gamma function were also established and collected in~\cite{refine-Ivady-gamma-PMD.tex, theta-new-proof.tex-BKMS, note-on-li-chen.tex}. See also~\cite[pp.~52\nobreakdash--57]{bounds-two-gammas.tex} and lots of references cited therein.
\end{rem}

\begin{rem}
It was proved once again in~\cite{CBerg, absolute-mon-simp.tex, compmon2} that the set of logarithmically completely monotonic functions is a subset of the completely monotonic functions. This implies that Theorem~\ref{bukac-app-LCM-thm} is not trivial.
\end{rem}

\begin{rem}
In~\cite{s-guo-ijpam, Guo-Qi-Srivastava2007.tex}, it was shown that the function
\begin{equation}
g_{\beta}(x)=\frac{e^x\Gamma(x+1)}{(x+\beta)^{x+\beta}}
\end{equation}
on the interval $(\max\{0,-\beta\},\infty)$ for $\beta\in\mathbb{R}$ is logarithmically completely monotonic if and only if $\beta\ge1$ and that the function $[g_{\alpha,\beta}(x)]^{-1}$ is logarithmically completely monotonic if and only if $\beta\le\frac12$. See also~\cite[pp.~53\nobreakdash--54, Section~5.6]{bounds-two-gammas.tex}. Motivated by this and Theorem~\ref{bukac-app-LCM-thm}, we would like to ask a question: How about the logarithmically complete monotonicity of the function
\begin{equation}
  G_{\lambda,\mu}(x)=\frac{e^{x}\Gamma(x+1)}{(x+\mu)^{x+\mu}}\exp\frac1{24(x+\lambda)}
\end{equation}
on the interval $(\max\{0,-\lambda,-\mu\},\infty)$? where $\lambda$ and $\mu$ are given real numbers.

\end{rem}


\begin{thebibliography}{99}

\bibitem{abram}
M. Abramowitz and I. A. Stegun (Eds), \textit{Handbook of Mathematical Functions with Formulas, Graphs, and Mathematical Tables}, National Bureau of Standards, Applied Mathematics Series \textbf{55}, 4th printing, with corrections, Washington, 1965.

\bibitem{Atanassov}
R. D. Atanassov and U. V. Tsoukrovski, \textit{Some properties of a class of logarithmically completely monotonic functions}, C. R. Acad. Bulgare Sci. \textbf{41} (1988), no.~2, 21\nobreakdash--23.

\bibitem{CBerg}
C. Berg, \textit{Integral representation of some functions related to the gamma function}, Mediterr. J. Math. \textbf{1} (2004), no.~4, 433\nobreakdash--439.

\bibitem{Bukac-Buric-Elezovic-MIA-2011}
J. Bukac, T. Buri\'c and N. Elezovi\'c, \textit{Stirling's formula revisited via some classical and new inequalities}, Math. Inequal. Appl. \textbf{14} (2011), no.~1, 235\nobreakdash--245.

\bibitem{psever1}
P. Cerone and S. S. Dragomir, \textit{Midpoint-type rules from an inequality point of view}, Handbook of Analytic-Computational Methods in Applied Mathematics, Editor: G. Anastassiou, CRC Press, New York, 2000, 135\nobreakdash--200.

\bibitem{psever2}
P. Cerone and S. S. Dragomir, \textit{Trapezoidal-type rules from an inequality point of view}, Handbook of Analytic-Computational Methods in Applied Mathematics, Editor: G. Anastassiou, CRC Press, New York, 2000, 65\nobreakdash--134.

\bibitem{Global-JAMMS-2008}
Ch.-P. Chen and F. Qi, \emph{The best bounds of the $n$-th harmonic number}, Glob. J. Appl. Math. Math. Sci. \textbf{1} (2008), no.~1, 41\nobreakdash--49.

\bibitem{absolute-mon-simp.tex}
B.-N. Guo and F. Qi, \textit{A property of logarithmically absolutely monotonic functions and the logarithmically complete monotonicity of a power-exponential function}, Politehn. Univ. Bucharest Sci. Bull. Ser. A Appl. Math. Phys. \textbf{72} (2010), no.~2, 21\nobreakdash--30.

\bibitem{refine-Ivady-gamma-PMD.tex}
B.-N. Guo and F. Qi, \textit{A refinement of a double inequality for the gamma function}, Publ. Math. Debrecen \textbf{79} (2011), in press.

\bibitem{harmonic-number-refine-chen.tex}
B.-N. Guo and F. Qi, \textit{Sharp bounds for harmonic numbers}, Appl. Math. Comput. \textbf{??} (2011), no.~??, in press; Available online at \url{http://dx.doi.org/10.1016/j.amc.2011.01.089}.

\bibitem{property-psi.tex}
B.-N. Guo and F. Qi, \textit{Some properties of the psi and polygamma functions}, Hacet. J. Math. Stat. \textbf{39} (2010), no.~2, 219\nobreakdash--231.

\bibitem{theta-new-proof.tex-BKMS}
B.-N. Guo and F. Qi, \textit{Two new proofs of the complete monotonicity of a function involving the psi function}, Bull. Korean Math. Soc. \textbf{47} (2010), no.~1, 103\nobreakdash--111; Available online at \url{http://dx.doi.org/10.4134/BKMS.2010.47.1.103}.

\bibitem{note-on-li-chen.tex}
B.-N. Guo, Y.-J. Zhang and F. Qi, \textit{Refinements and sharpenings of some double inequalities for bounding the gamma function}, J. Inequal. Pure Appl. Math. \textbf{9} (2008), no.~1, Art.~17; Available online at \url{http://www.emis.de/journals/JIPAM/article953.html?sid=953}.

\bibitem{s-guo-ijpam}
S. Guo, \textit{Some classes of completely monotonic functions involving the gamma function}, {Internat. J. Pure Appl. Math.} \textbf{30} (2006), no.~4, 561\nobreakdash--566.

\bibitem{Guo-Qi-Srivastava2007.tex}
S. Guo, F. Qi and H. M. Srivastava, \textit{Necessary and sufficient conditions for two classes of functions to be logarithmically completely monotonic}, Integral Transforms Spec. Funct. \textbf{18} (2007), no.~11, 819\nobreakdash--826; Available online at \url{http://dx.doi.org/10.1080/10652460701528933}.

\bibitem{Hoorfar-Qi.tex}
A. Hoorfar and F. Qi, \textit{Some new bounds for Mathieu's series}, Abstr. Appl. Anal. \textbf{2007} (2007), Article ID 94854, 10 pages; Available online at \url{http://dx.doi.org/10.1155/2007/94854}.

\bibitem{magnus}
W. Magnus, F. Oberhettinger and R. P. Soni, \textit{Formulas and Theorems for the Special Functions of Mathematical Physics}, Springer, Berlin, 1966.

\bibitem{mpf-1993}
D. S. Mitrinovi\'c, J. E. Pe\v{c}ari\'c and A. M. Fink, \textit{Classical and New Inequalities in Analysis}, Kluwer Academic Publishers, Dordrecht/Boston/London, 1993.

\bibitem{e^x-beograd}
F. Qi, \textit{A method of constructing inequalities about $e^x$}, Univ. Beograd. Publ. Elektrotehn. Fak. Ser. Mat. \textbf{8} (1997), 16\nobreakdash--23.

\bibitem{mathieu-rostock}
F. Qi, \textit{An integral expression and some inequalities of Mathieu type series}, Rostock. Math. Kolloq. \textbf{58} (2004), 37\nobreakdash--46.

\bibitem{bounds-two-gammas.tex}
F. Qi, \textit{Bounds for the ratio of two gamma functions}, J. Inequal. Appl. \textbf{2010} (2010), Article ID 493058, 84~pages; Available online at \url{http://dx.doi.org/10.1155/2010/493058}.

\bibitem{compmon2}
F. Qi and Ch.-P. Chen, \textit{A complete monotonicity property of the gamma function}, J. Math. Anal. Appl. \textbf{296} (2004), no.~2, 603\nobreakdash--607.

\bibitem{mathieuijms.tex}
F. Qi, Ch.-P. Chen and B.-N. Guo, \textit{Notes on double inequalities of Mathieu's series}, Int. J. Math. Math. Sci. \textbf{2005} (16) (2005), 2547\nobreakdash--2554; Available online at \url{http://dx.doi.org/10.1155/IJMMS.2005.2547}.

\bibitem{qi-cui-jmaa}
F. Qi, R.-Q. Cui, Ch.-P. Chen and B.-N. Guo, \textit{Some completely monotonic functions involving polygamma functions and an application}, J. Math. Anal. Appl. \textbf{310} (2005), no.~1, 303\nobreakdash--308.

\bibitem{Extended-Binet-remiander-comp.tex}
F. Qi and B.-N. Guo, \textit{Some properties of extended remainder of Binet's first formula for logarithm of gamma function}, Math. Slovaca \textbf{60} (2010), no.~4, 461\nobreakdash--470; Available online at \url{http://dx.doi.org/10.2478/s12175-010-0025-7}.

\bibitem{difference-hermite-hadamard.tex}
F. Qi, Z.-L. Wei and Q. Yang, \textit{Generalizations and refinements of Hermite-Hadamard's inequality}, Rocky Mountain J. Math. \textbf{35} (2005), no.~1, 235--251.

\bibitem{correction-to-sandor.tex-ijams}
F. Qi and M.-L. Yang, \textit{Comparisons of two integral inequalities with Hermite-Hadamard-Jensen's integral inequality}, Internat. J. Appl. Math. Sci. \textbf{3} (2006), no.~1, 83\nobreakdash--88.

\bibitem{correction-to-sandor.tex-octogon}
F. Qi and M.-L. Yang, \textit{Comparisons of two integral inequalities with Hermite-Hadamard-Jensen's integral inequality}, Octogon Math. Mag. \textbf{14} (2006), no.~1, 53\nobreakdash--58.

\bibitem{Schilling-Song-Vondracek-2010}
R. L. Schilling, R. Song and Z. Vondra\v{c}ek, \textit{Bernstein Functions}, de Gruyter Studies in Mathematics \textbf{37}, De Gruyter, Berlin, Germany, 2010.

\bibitem{Sevli-Batir-MCM-2011}
H. \c{S}evli and N. Bat{\i}r, \textit{Complete monotonicity results for some functions involving the gamma and polygamma functions}, Math. Comput. Modelling \textbf{53} (2011), 1771\nobreakdash--1775; Available online at \url{http://dx.doi.org/10.1016/j.mcm.2010.12.055}.

\bibitem{widder}
D. V. Widder, \textit{The Laplace Transform}, Princeton University Press, Princeton, 1946.

\end{thebibliography}
\end{document}